\documentclass[11pt,leqno]{amsart}
\usepackage{amsthm,amsfonts,amssymb,amsmath,oldgerm}
\numberwithin{equation}{section}
\usepackage{fullpage}
\usepackage{color}
\usepackage{calrsfs}
\DeclareMathAlphabet{\pazocal}{OMS}{zplm}{m}{n}

\newcommand{\rmc}{{\mathrm{c}}}
\newcommand{\rmh}{{\mathrm{h}}}
\newcommand{\obr}{\overline{\mathbf{r}}}
\newcommand{\mez}{{\frac{1}{2}}}
\newcommand{\tA}{\wti{A}}
\newcommand{\tbV}{\wti{\bV}}
\newcommand{\tu}{\wti{u}}
\newcommand{\tv}{\wti{v}}
\newcommand{\sR}{\hat{\mathcal{R}}}

\def\eps{\varepsilon }

\newcommand\R{\mathbb R}

\def\eps{\varepsilon}

\newcommand\br{\begin{remark}}
\newcommand\er{\end{remark}}
\newcommand\bp{\begin{pmatrix}}
\newcommand\ep{\end{pmatrix}}
\newcommand{\be}{\begin{equation}}
\newcommand{\ee}{\end{equation}}
\newcommand\ba{\begin{equation}\begin{aligned}}
\newcommand\ea{\end{aligned}\end{equation}}

\newcommand{\bap}{\begin{app}}
\newcommand{\eap}{\end{app}}
\newcommand{\begs}{\begin{exams}}
\newcommand{\eegs}{\end{exams}}
\newcommand{\beg}{\begin{example}}
\newcommand{\eeg}{\end{exaplem}}
\newcommand{\bpr}{\begin{proposition}}
\newcommand{\epr}{\end{proposition}}
\newcommand{\bt}{\begin{theorem}}
\newcommand{\et}{\end{theorem}}
\newcommand{\bc}{\begin{corollary}}
\newcommand{\ec}{\end{corollary}}
\newcommand{\bl}{\begin{lemma}}
\newcommand{\el}{\end{lemma}}
\newcommand{\bd}{\begin{definition}}
\newcommand{\ed}{\end{definition}}
\newcommand{\brs}{\begin{remarks}}
\newcommand{\ers}{\end{remarks}}

\newcommand{\A }{\mathcal{A}}

\newcommand{\RR}{{\mathbb R}}

\newcommand{\pa}{{\partial}}

\newcommand{\rms}{{\mathrm{s}}}
\newcommand{\rmu}{{\mathrm{u}}}

\newcommand{\const}{\text{\rm constant}}
\newcommand{\Id}{{\rm Id }}
\newcommand{\im}{{\rm im }}
\newcommand{\Range}{{\rm Range }}

\newcommand{\sgn}{\text{\rm sgn}}
\newtheorem{theorem}{Theorem}[section]
\newtheorem{proposition}[theorem]{Proposition}
\newtheorem{corollary}[theorem]{Corollary}
\newtheorem{lemma}[theorem]{Lemma}

\theoremstyle{remark}
\newtheorem{remark}[theorem]{Remark}
\theoremstyle{definition}
\newtheorem{definition}[theorem]{Definition}

\newtheorem{example}[theorem]{Example}

\newtheorem{obs}[theorem]{Observation}
\newcommand\cJ{{\mathcal J}}

\newcommand\cE{{\mathcal E}}

\newcommand\cQ{{\mathcal Q}}

\newcommand\cM{{\mathcal M}}

\newcommand\bH{{\mathbb H}}

\newcommand\bV{{\mathbb V}}

\newcommand{\f}{\frac}

\newcommand{\wti}{\widetilde}

\newcommand{\bu}{\mathbf{u}}
\newcommand{\bw}{\mathbf{w}}
\newcommand{\bz}{\mathbf{z}}
\newcommand{\obu}{\overline{\mathbf{u}}}

\newcommand{\dom}{\text{\rm{dom}}}

\newcommand{\beq}{\begin{equation}}
\newcommand{\eeq}{\end{equation}}
\title{Invariant manifolds for a class of degenerate evolution equations and structure of kinetic shock layers
}

\author{Kevin Zumbrun}
\address{Indiana University, Bloomington, IN 47405}
\email{kzumbrun@indiana.edu}
\thanks{Research of K.Z. was partially supported
under NSF grant no. DMS-0300487}

\begin{document}

\begin{abstract}
We describe recent results with A. Pogan developing dynamical systems tools for a class of
degenerate evolution equations arising in kinetic theory, including 
the steady Boltzmann and BGK equations.
These yield information on structure of large- and small-amplitude kinetic shocks,
the first steps in a larger program toward time-evolutionary stability and asymptotic behavior.
\end{abstract}

\maketitle

\tableofcontents
\section{Introduction}
In these notes, we describe recent results \cite{PZ1,PZ2} with Alin Pogan developing a set of dynamical
systems tools suitable for the study of existence and structure of shock and boundary layer solutions
arising in Boltzmann's equation and related kinetic models.  These represent the first steps in a larger program to develop 
dynamical systems methods like those used in the study of finite-dimensional viscous and relaxation shocks in 
\cite{GZ,MasZ2,Z2,Z3,Z4,Z5,ZH,ZS}, suitable for treatment of one- and multi-dimensional stability of large-amplitude kinetic shock and boundary layers.

\subsection{Equations and assumptions}
Our goal is the study of shock or boundary layer solutions
\be\label{wave}
\bu(x,t)= \check {\bu}(x), \qquad \lim_{x\to\pm\infty} \check {\bu}(x)=\bu^\pm,
\ee
of kinetic-type relaxation systems
\begin{equation}\label{Relax}
	A^0 \bu_t + A\bu_x = Q(\bu)
\end{equation}
on a Hilbert space $\bH$, where $A^0$ and $A$ are constant bounded 
linear operators, and $Q$, the {\it collision operator}, is a bounded bilinear map.
This leads  us to the study of the associated {\it steady equation}
\be\label{steady}
A\bu'=Q(\bu).
\ee

Following \cite{MZ,PZ1,PZ2}, we make the following structural assumptions.
\smallskip

\noindent{\bf Hypothesis (H1)} (i) The linear operator $A$ is
bounded, self-adjoint, and one-to-one
on the Hilbert space $\bH$, but {\it not boundedly invertible}. (ii) There exists $\bV$ a proper, closed subspace of $\bH$ with $\dim\bV^\perp<\infty$ and $B:\bH\times\bH\to\bV$ is a bilinear, symmetric, continuous map such that $Q(\bu)=B(\bu,\bu)$.

\smallskip

\noindent{\bf Hypothesis (H2)} There exist an equilibrium $\obu\in\ker Q$ satisfying
\begin{enumerate}
\item[(i)] $Q'(\obu)$ is self-adjoint and $\ker Q'(\obu)=\bV^\perp$;
\item[(ii)] There exists $\delta>0$ such that $Q'(\obu)_{|\bV}\leq -\delta I_{\bV}$;
\end{enumerate}

The class of system so described includes in particular our 
main example, of {\it Boltzmann's equation} with hard-sphere potential, written in appropriate coordinates
\cite{MZ}; see Section \ref{s:reduction}.
As regards \eqref{steady}, the main novelty is that $A$ by (H1)(i) has an {\it essential singularity}, i.e., essential
spectrum at the origin,
hence \eqref{steady} is a {\it degenerate evolution equation} to which invariant manifold results
of standard dynamical systems theory do not immediately apply.
Our purpose here is precisely the construction of invariant manifolds for the class of degenerate equations 
\eqref{steady} satisfying (H1)-(H2),
and the application of these tools toward existence and structure of kinetic shock and boundary layers.

\br\label{kawrmk}
We do not assume as in \cite{MZ} the ``genuine coupling'' or ``Kawashima'' condition that
no eigenvector of $A$ lie in the kernel of $Q'(\obu)$.
The assumption $A$ one-to-one implies (trivially) the weaker condition, sufficient for our analysis,
that no zero eigenvector of $A$ lie in the kernel of $Q'(\obu)$.
\er

\subsection{Chapman-Enskog expansion and canonical form}\label{s:ce}
Our starting point is the formal {\it Chapman-Enskog} expansion designed to approximate near-equilibrium flow \cite{L}.
Near $\obu$, (H1)-(H2) yields by the Implicit Function Theorem 
existence of a (Fr\'echet) $C^\infty$ manifold of equlibria 
\be \label{eq}
\cE= \ker Q, \qquad \dim \cE=\dim \bV^\perp=:r,
\ee
tangent to $\bV^\perp$ at $\obu$, expressible in coordinates $\bw:=\bu-\obu$ as a $C^\infty$
graph 
\be\label{v*}
v_*:\bV^\perp \to\bV.
\ee
Denote $u=P_{\bV^\perp}\bu$, $v=P_{\bV}\bu$, where $P_{\bV^\perp}$ and $P_{\bV}$ are the orthogonal
projections onto $\bV^\perp$ and $\bV$ associated with the decomposition $\bH=\bV^\perp \oplus \bV$.
The second-order Chapman-Enskog approximation, or ``hydrodynamic limit,'' of \eqref{Relax} is then
 $h_*(u)_t + f_*(u)_x=D_*u_{xx}$, with associated steady equation 
\be\tag{${\rm CE}$}\label{ce2}
 f_*(u)_x=D_*u_{xx},
\ee
 where $ h_*(u):= P_{\bV^\perp} A^0 (u^T, v_*(u)^T)^T$ and
 \be\label{D}
  f_*(u):= P_{\bV^\perp} A (u^T, v_*(u)^T)^T,
  \quad D_*:=A_{12} E^{-1}  A_{12}^T,
\ee
with $A_{12}:=P_{\bV^\perp}A P_{\bV}$ and $E:= Q'(\obu)_{|\bV}$. See \cite{L,MZ,PZ2} for further details.

From (H1)(ii), $P_{\bV^\perp} (A \bu)'= P_{\bV^\perp}Q\equiv 0$. Integrating, we find that
\eqref{steady} admits a {\it conservation law}
\be\label{cons}
P_{\bV^\perp} A \bu\equiv q=\const.
\ee
By the definition of $f_*$, $v_*$, equilibria $\bu_\pm= (u^T, v_*(u)^T)^T_\pm$ satisfy the
{\it Rankine-Hugoniot} condition
\be\tag{RH}\label{rh}
f_*(u_+)=f_*(u_-)=q
\ee
associated with viscous shock profiles of the Chapman-Enskog system \eqref{ce2}, giving a rigorous connection
at the inviscid level between shock or boundary layer profiles of the two systems \eqref{Relax} and \eqref{ce2}.
A further connection, between the types of the equilibria $\obu=(\bar u^T, v_*(\bar u)^T)^T$
and  $\bar u$ with respect to their associated flows, is given by the following 
key observation proved in Section \ref{s:reduction}.

\begin{lemma}\label{l:canon} System \eqref{steady} may, by an invertible change of coordinates, be put
in canonical form
\ba\label{canon}
w_c'&=Jw_c + \tilde Q_c(w_c,w_h)\\
\Gamma_0 w_h' &= -w_h + \tilde Q_h(w_c,w_h),
\ea
$w_c$ and $w_h$ parametrizing center and hyperbolic (i.e., stable/unstable) subspaces, $\dim w_c=m+r$,
$m=\dim\ker f_*'(\bar u)$, $r=\dim \bV^\perp$,
where $J=\bp 0 & I_m & 0 \\ 0&0 & 0\\ 0&0&0 \ep$ is a nilpotent block-Jordan form, 
$\Gamma_0$ is a constant, bounded symmetric operator, and $\tilde Q_j(w_c,w_h)=O(|w_c,w_h|^2)$.
In case $m=0$, $J, \tilde Q_c\equiv 0$.
\end{lemma}

One may compute that the perturbation equations for \eqref{ce2} about $\bar u$ have the same canonical
form (noting $f_*'(\bar u)= P_{\bV^\perp} AP_{\bV^\perp}$, $D_*$ symmetric) with 
$\Gamma_0$ finite-dimensional, invertible \cite{MaP,Pe}.

\subsection{Dichotomies vs. direct $L^p$ estimate} \label{s:dichotomy}
Lemma \ref{l:canon} effectively reduces the study of near-equilibrium flow of \eqref{steady}
to understanding the hyperbolic operator $(\Gamma_0\partial_x + \Id)$,
specifically, obtaining bounds on solutions of the degenerate inhomogeneous linear evolution system
\be\label{inhom}
(\Gamma_0\partial_x + \Id) w_c=g,
\ee
where $\Gamma_0$ is bounded, symmetric, and one-to-one, but (by (H1)) {\it not boundedly invertible:}
formally,
\be\label{formal}
(\partial_x + \Gamma_0^{-1}) w_c= \tilde g,
\ee
where $\Gamma_0^{-1}$ is an unbounded self-adjoint operator and $\tilde g:= \Gamma_0^{-1} g$.
As $\Gamma_0$ is indefinite, \eqref{formal} is {\it ill-posed} with respect to the Cauchy problem, 
featuring unbounded growth in both directions.

Ill-posed equations, and the derivation of associated resolvent bounds, have been treated in a variety
of contexts via {\it generalized exponential dichotomies}:
for example, modulated waves on cylindrical domains 
\cite{PSS,SS1,SS2}, Morse theory \cite{AM1,AM2,RobSal}, PDE Hamiltonian systems
\cite{BjornSand}, and the functional-differential equations of mixed type \cite{Mallet-Paret}.
It is not difficult to see, either by spectral decomposition of $\Gamma_0$, or by Galerkin approximation, 
that $(\partial_x + \Gamma_0^{-1})$ generates a {\it stable bi-semigroup} \cite{BGK,LP2}, 
the infinite-dimensional analog of an exponential dichotomy, 
that is, there exist bounded projections on whose range the homogeneous flow
is exponentially decaying in forward/backward direction, in this case with rate $|\Gamma_0|_{\bH}^{-1}$,
where $|\cdot|_{\bH}$ denotes operator norm; see \cite{PZ1} for details.

This, however, yields only $ \|u\|\leq C\|\tilde g\|=\|\Gamma_0^{-1}g\|, $
the intervention of the unbounded operator $\Gamma_0^{-1}$ making these bounds useless for our analysis.
Thus, the present problem differs from the above-mentioned ones in that
{\it exponential dichotomies are inadequate to bound the resolvent}
$(\Gamma_0 \partial_x+\Id)^{-1}$.
Indeed, we have the following striking result obtained by direct estimate in Section \ref{s:linear},
showing that our situation is one of {\it maximal regularity}.
In this sense, our analysis is related in flavor to construction of center manifolds for quasilinear systems;
see \cite{HI,Mi}, and references therein.

\begin{lemma}\label{l:linbd}
	Assuming (H1)-(H2), $|(\Gamma_0 \partial_x + \Id)^{-1}|_{L^p(\R)}<\infty$ for $1<p<\infty$,
		but \emph{not} for $p=1,\infty$.
\end{lemma}

An important consequence is that usual weighted $L^\infty$ constructions of invariant manifolds are unavailable.
We work instead in {\it weighted $H^1$ spaces}, 
with accompanying new technical issues.

\subsection{Results}
We are now ready to state our main results.
Assuming (H1)-(H2), from \eqref{canon} and symmetry of $\Gamma_0$ we readily obtain a decomposition
$\bH=\bH_\rmc\oplus \bH_\rmc\oplus \bH_\rmu$ of $\bH$ into stable, center, and unstable subpaces invariant
under the homogeneous linearized flow of \eqref{steady} about the equilibrium $\obu$.
Let $H^1_\eta(\R,\bH)$ denote the space of functions bounded in the exponentially weighted $H^1$ norm
\be\label{norm}
\|f\|_{H^1_\eta(\R,\bH)}:= \|e^{\eta \langle \cdot \rangle}f(\cdot)\|_{L^2(\R,\bH)}
+\|e^{\eta \langle \cdot \rangle}f'(\cdot)\|_{L^2(\R,\bH)},
\ee
where $\langle x\rangle:=(1+|x|^2)^{1/2}$ and $\eta\in \R$ may be positive or negative according to our needs.
Following \cite{LP2}, we define solutions of \eqref{steady} using Lemma \ref{l:linbd} as 
$H^1_{loc}$ solutions of the fixed-point equation $w_h=(\Gamma_0\partial_x+\Id)^{-1}g_c(w)$ and the finite-dimensional ODE
$(\partial_x-J)w_c=g_c(w)$ in $w_c$; see \cite{PZ1,PZ2}.

\subsubsection{$H^1$ stable manifold and exponential decay of large-amplitude shock and boundary layers}\label{s:invariant}
Our first observation is that for singular $\Gamma_0$ the $H^1$ stable subspace of \eqref{canon}, defined as the trace at $x=0$ of 
solutions $w_h$ bounded in $H^1(\R^+,\bH)$, {\it is a dense proper subspace of $\bH_\rms$,}
related to the domain of the generator $\Gamma_0^{-1}$ of the bi-semigroup associated with
homogeneous linearized flow.

\begin{lemma}\label{l:dom}
	Assuming (H1)-(H2),
	the $H^1$ stable subspace of the linearized equations of \eqref{steady} about $\obu$ 
	(equivalently, the linearization of \eqref{canon} about $0$)
is $\dom(|\Gamma_0|^{-1/2})\cap \bH_\rms \subset \bH_\rms$.
\end{lemma}

\begin{proof}
The $H^1$ stable subspace consists of $f\in \bH_\rms$ such that
$ \int_0^\infty \langle \partial_x e^{\Gamma_0^{-1} x}f, \partial_x e^{\Gamma_0^{-1} x}f\rangle  dx <\infty,  $
or, equivalently,
$ -(1/2)\int_0^\infty \partial_x \langle   e^{\Gamma_0^{-1} x} |\Gamma_0|^{-1/2}f,  
e^{\Gamma_0^{-1} x}|\Gamma_0|^{-1/2}f\rangle  dx <\infty.  $
Integrating, and observing that the boundary term at infinity vanishes, gives condition
$\langle |\Gamma_0|^{-1/2}f , |\Gamma_0|^{-1/2}f \rangle  <\infty.  $
Alternatively, this may be deduced by spectral decomposition of $\Gamma_0$ and direct computation \cite{PZ1}.
\end{proof}

We have accordingly the following modification of the usual stable manifold theorem.

\begin{theorem}\label{t1.3}
Assuming (H1)-(H2), for any $0<\alpha<\tilde \nu <\nu:=|\Gamma_0|_{\bH}^{-1}$, 
there exists a local stable manifold $\cM_\rms$ near $\obu$, 
expressible in coordinates $w=\bu-\obu$ as a $C^1$ embedding tangent to $\bH_s$ 
of $\dom(\Gamma_0^{-1/2})\cap \bH_\rms$ with (graph) norm induced by $\Gamma_0^{-1/2}$ into $\bH$, 
locally invariant under the flow of \eqref{steady}, containing the orbits of all solutions $w$ with
$ H^1_{\alpha}(\RR_+,\bH)$ norm sufficiently small,
with solutions $w$ initiating in $\cM_\rms$ at $x=0$ lying in $ H^1_{\tilde \nu}(\RR_+,\bH)$.
In case $\det f_*'(u_+)\neq 0$, $\alpha$ may be taken to be zero.
\end{theorem}

We obtain as a consequence exponential decay of noncharacteristic shock or boundary layers.

\begin{corollary}\label{c5.9}
	Assuming (H1)-(H2), let $\obu$ be a noncharacteristic equilibrium
	in the sense of \eqref{ce2}, $\det f_*'(\bar u)\neq 0$, and $\tilde \nu <\nu= 1/|\Gamma_0|_{\bH}$.
	Then, for any solution $\check{\bu}$ of \eqref{steady} converging to $\obu$ as $x\to +\infty$ in the sense that
	$\check {\bu}- \obu$ is eventually bounded in $H^1([x,\infty),\bH)$, we have \emph{exponential decay}:
		\be\label{expdecay}
		|\check u-\obu|_{\bH}(x)\lesssim e^{-\tilde \nu x}
		\quad \hbox{\rm as $x\to +\infty$.}
		\ee
\end{corollary}

\subsubsection{Center manifold and structure of small-amplitude shock layers}
We have, similarly, the following modification of the usual center manifold theorem (cf. \cite{B,HI,Z1,Zode}).

\begin{theorem}\label{t1.1}
Let $\obu$ be an equilibrium satisfying (H1)-(H2).  Then, for any integer $k\geq2$
there exists local to $\obu$ a $C^k$ center manifold $\cM_\rmc$, tangent at $\obu$ to $\bH_\rmc$,
expressible in coordinates $\bw:=\bu-\obu$ as a $C^k$ graph $\cJ_\rmc:\bH_\rmc\to\bH_\rms\oplus\bH_\rmu$,
that is locally invariant under the flow of \eqref{steady} and contains all solutions
that remain sufficiently close to $\obu$ in forward and backward $x$.
Moreover, $\cM_\rmc$ has the \emph{$H^1$ exponential approximation property}: for any $0<\tilde \nu< \nu=1/|\Gamma_0|_{\bH}$,
a solution ${\bu}$ of \eqref{steady} 
with $\|{\bu}-\obu\|_{H^1_{-\alpha}\cap L^\infty([M,\infty),\bH)}$ and $\alpha>0$ sufficiently small approaches a solution 
$\bz$ with orbit lying in $\cM_\rmc$ as $x\to +\infty$ at exponential rate $\|{\bu}-\bz\|_{\bH}\lesssim e^{-\tilde \nu x}$,
with also $\|{\bu}-\bz\|_{H^1_{\tilde \nu}([M,\infty),\bH)}<\infty$.
\end{theorem}

Here, the only difference from the standard center manifold theorem \cite{B}
is the weakened, $H^1$, version of the exponential approximation property.
For applications involving normal form reduction, they are essentially equivalent; in particular, the
formal Taylor expansion for center graph $w_h=\Xi(w_c)$ may be computed to arbitrary order 
in coordinates \eqref{canon} by successively matching terms of increasing order in 
the defining relation $\Gamma_0 \Xi(w_c)'  =-\Xi(w_h)+\tilde Q_h$, or equivalently
 $ \Xi(w_c)= -\Gamma_0 \Xi'(w_c)(Jw_c + \tilde Q_c)+ \tilde Q_h  $,
 exactly as in the usual (nonsingular $A$, $\Gamma_0$) case \cite{Carr,HI}.

\br\label{cmdecay}
In the noncharacteristic case, the center manifold, by dimensional count and the fact that it must contain all local equilibria,
is uniquely determined as the manifold of equilibria $\cE$.
In this case, the exponential approximation property improves slightly the result of Corollary \ref{c5.9},
yielding that solutions $\check {\bu}$ of \eqref{steady} lying sufficiently close to $\obu$ in $L^\infty(\R^+,\bH)$ 
and sufficiently slowly {exponentially growing} in $H^1$, converges to an equilibrium
at exponential rate $e^{-\tilde \nu x}$, $0<\tilde \nu<1/|\Gamma_0|_{\bH}$.
\er

Denote the characteristics of Chapman-Enskog system \eqref{ce2}, or eigenvalues of $f_*'(u)$, 
by 
$$
\lambda_1(u)\leq \dots \leq \lambda_r(u).
$$
The {\it noncharacteristic case} $f_*'(\bar u)\neq 0$ is the case that no characteristic velocity $\lambda_j(\bar u)$
vanishes, in which case,
by the Inverse Function Theorem, the Rankine-Hugoniot equations \eqref{rh} admit a single nearby solution for each value of $q$,
hence no local shock connections occur.
To study small-amplitude shock profiles, we focus therefore on the {\it characteristic case} $f_*'(\bar u)=0$, specifically
on the generic case that $\lambda_j(\bar u)=0$ for a single characteristic velocity $\lambda_p$, 
with associated unit eigenvector $\obr$, that is {\it genuinely nonlinear} in the sense of Lax \cite{La,Sm}:
\be\tag{GNL}\label{gnl}
\Lambda:=\obr\cdot f_*''(\bar u)(\obr,\obr)\neq 0.
\ee

In this case, it is well known \cite{La,Sm,MaP} that there exists a family of small-amplitude shock profiles 
$\check{\bar u}$ of \eqref{ce2} connecting
endstates $\bar u_\pm\to \bar u$, with $(\bar u_+-\bar u_-)$ lying in approximate direction $\obr$, 
with $\lambda:=\lambda_p(\check{\bar u})$ satisfying an approximate Burgers equation
\be\label{burgers}
\delta \lambda'=-\eps^2 + \lambda^2/2 + O(|\eps, \lambda|^3),
\ee
$\Lambda$ as in \eqref{gnl}, $\eps>0$ parametrizing amplitude,
provided there holds the {\it stable viscosity criterion}
$\delta:=\obr \cdot D_* \obr>0,$
as may be readily seen to hold for $D_*$ using \eqref{D} and (H1) (cf. Rmk. \ref{kawrmk}).

Our final result gives a corresponding characterization of small-amplitude kinetic shocks of \eqref{steady}
bifurcating from a simple genuinely nonlinear eigenvalue of $f_*'(\bar u)$.  The complementary case of bifurcation from a
multiple, linearly degenerate eigenvalue of $f_*'(\bar u)$ \cite{La,Sm} is treated also in \cite[Thm. 1.5]{PZ2} (not stated here);
in that case, no nontrivial shock or boundary layer connections exist.

\begin{corollary}\label{c1.3}
Let $\obu$ be an equilibrium satisfying (H1)-(H2)
in the characteristic case \eqref{gnl}, $\lambda_p(\bar u)=0$ a simple eigenvalue, and $k$ an integer $\geq 2$.
Then, local to $\obu$, $\bar u$, each pair of points $u_\pm$ satisfying the Rankine-Hugoniot condition \eqref{rh}
has a corresponding viscous shock solution $u_{CE}$ of \eqref{ce2} and
relaxation shock solution $\bu_{REL}=(u_{REL},v_{REL})$ of \eqref{steady}, satisfying for all $j\leq k-2$:
\begin{equation}\label{finalbds}
\begin{aligned}
	\big|\partial_x^j ( u_{REL}- u_{CE})\big| &\le C \eps^{j+2}e^{-\mu  \eps|x|},\\
	\big|\partial_x^j \big(v_{REL}-v_*(u_{CE})\big)\big| &\le C \eps^{j+2}e^{-\mu  \eps|x|},\\
|\partial_x^j (u_{REL}-u_\pm)|&\le C \eps^{j+1}e^{-\mu \eps|x|}, \quad x\gtrless 0,\\
\end{aligned}
\end{equation}
$\mu>0$, $C>0$,  $\eps:=|u_+-u_-|$, 
unique up to translation, 
with $\lambda_p(u_{REL})$ and $\lambda_p(u_{CE})$ both satisfying approximate Burgers equations \eqref{burgers}:
in particular, both monotone decreasing in $x$.
\end{corollary}

\subsection{Discussion and open problems}\label{s:discussion}
Corollary \ref{c1.3} recovers under slightly weakened assumptions, the result of \cite[Prop. 5.4]{MZ},
which, applied to Boltzmann's equation, in turn recovers and sharpens
the fundamental result \cite{CN} of existence of small-amplitude Boltzmann shocks with standard, square-root 
Maxwellia-weighted $L^2$ norm in velocity \cite{G}.
With further effort, one may show \cite[Prop. 1.8]{PZ2} (not stated here) that the center manifold of Theorem \ref{t1.1},
hence also the small-amplitude shock profiles obtained of Corollary \ref{c1.3},
are contained in a stronger space of near Maxwellian-weighted $L^2$ norm in velocity, 
recovering the strongest current existence result for Boltzmann shocks \cite[Thm. 1.1]{MZ}, 
plus the additional dynamical information of \eqref{burgers} and monotonicity of $\lambda_p(u_{REL}(x))$-
neither available by the Sobolev-based fixed point iteration arguments of \cite{CN,MZ}.

To our knowledge, {Theorems \ref{t1.3} and \ref{t1.1} are the first results on existence of invariant manifolds} for any
system of form \ref{Relax}, (H1)-(H2) in either Hilbert or Banach space setting,
{in particular for the steady Boltzmann equation with hard sphere potential}.
Liu and Yu \cite{LiuYu} have studied existence of invariant manifolds for Boltzmann's equation
in a weighted $L^\infty$ (in both velocity and $x$) Banach space setting, using
rather different methods of time-regularization and detailed pointwise bounds,
pointing out that monotonicity of $\lambda_p(\bar u)$ follows from center manifold reduction
and describing physical applications of center manifold theory
to condensation and subsonic/supersonic transition in Milne's problem.
However, their claimed linearized bounds, based on exponential dichotomies, 
hence also their arguments for existence of invariant manifolds, were incorrect \cite{Z6}; see Remark \ref{LYrmk}.
%CHANGED, softened:
Our results among other things repair this gap, validating their larger program/physical conclusions.
%Our results repair this gap, thus validating their larger program/physical conclusions.

A longer term program is to develop further dynamical systems tools for kinetic systems \eqref{Relax} with 
structure (H1)-(H2), sufficient to treat {\it time-evolutionary stability} of shock and boundary layers by the 
methods used for viscous/relaxation shocks in \cite{GZ,MasZ2,Z2,Z3,Z4,Z5,ZH,ZS}.
Besides unification/simplification, this approach has the advantage of applying in principle to multi-dimensional and/or
large-amplitude waves, each of these long-standing open problems in the area.

These techniques have the further advantages of separating the issues of existence, spectral stability, 
and linearized/nonlinear stability, with the first two often treated by a combination of analytical and numerical methods,
up to and including (see, e.g., \cite{Ba,BZ}) interval arithmetic-based rigorous numerical proof.
The development of numerical and or analytical methods for the treatment of existence of large-amplitude
kinetic shocks we regard as a further, very interesting open problem.

Indeed, the {\it structure problem} discussed by Truesdell, Ruggeri, Boillat, and others, of existence and description of large-amplitude Boltzmann shocks, is perhaps {\it the} fundamental open problems in the theory, and
one of the main motivations for their study.
As discussed, e.g., in \cite{BR}, Navier-Stokes theory
well-describes behavior of shocks of Mach number $M\lessapprox 2$, but inaccurately predicts shock width/structure
at large Mach numbers; by contrast, Boltzmann's equation (numerically and via various formal approximations) appears to 
match experiment in the large-$M$ regime.

\section{Reductions and main example}\label{s:reduction}
We begin by carrying out various reductions, first from Boltzmann's equation to the abstract form \eqref{steady}, (H1)-(H2),
then the abstract equation to the canonical form \eqref{canon}.

\subsection{Boltzmann's equation}\label{s:boltz}
(Following \cite{MZ})
Our main interest is Boltzmann's equation with hard-sphere potential (or Grad hard cutoff potential as in \cite{CN}):
\be\label{Boltz}
f_t + \xi_1 \partial_x f=  \cQ(f,f),
\ee
where $f(x,t,\xi)\in \R$ is the distribution of velocities $\xi\in \R^3$ at $x$, $t\in \R$, and
\begin{equation}
\label{colop}
\cQ (g, h) := \int \big( g( \xi') h (\xi'_*)  -  g(\xi) h(\xi_*) \big) C(\Omega, \xi - \xi_*) d \Omega d\xi_* 
\end{equation}
is the collision operator, with collision kernel $C (\Omega, \xi) = \big| \Omega \cdot \xi \big|$ for hard-sphere case. 

The space of \emph{collision invariants} $\langle \psi \rangle $, $ \int_{\R^3}\psi(\xi)\cQ(g, g)( \xi)d\xi \equiv 0   $,
of \eqref{Boltz} is spanned by
\begin{equation}
\label{defR}
Rf: = \int \Psi (\xi ) f( \xi ) d\xi   \in \R^5 , \qquad 
\Psi (\xi ) = (1 ,\xi_1 , \xi_2 , \xi_3 , \mez | \xi|^2 )^T.
\end{equation}
(Here, we are assuming that distributions $f(x,t,\cdot)$  are confined to a space $\bH$
to be specified later such that the integral converges.)
The associated {macroscopic (fluid-dynamical) variables} are 
\begin{equation}\label{fvar}
\rmu  := Rf=: (\rho , \rho v_1 , \rho v_2 ,  \rho v_3 , \rho E )^T,
\end{equation} 
where $\rho$ denotes density, $v=(v_1,v_2,v_3)$ velocity,
$E= e +  \mez | v|^2 $ total energy density, and $e$ internal energy density.
The set of \emph{equilibria} ($\ker \cQ$) consists of the \emph{Maxwellian distributions}:
	\begin{equation}\label{max}
M_\rmu(\xi)  = \frac{\rho}{\sqrt{(4\pi e/3)^3}}  
e^{-  \frac{| \xi - v |^2}{4 e/3} } . 
\end{equation}

\subsubsection{Symmetry, boundeness, and spectral gap}\label{s:sym}
{Boltzmann's $H$-theorem} \cite{G,Gl,Ce} (equivalent to existence of a thermodynamical entropy in the sense of
\cite{CLL}) asserts the variational principle
$$
\int \log f \cQ(f,f) d\xi \le 0 ,
$$
with equality on the set of Maxwellians $\underline M$.
Taylor expanding about a local maximum $\underline M$, we obtain symmetry and nonnegativity of the Hessian
$ \int \underline M^{-1} ( \partial \cQ|_{\underline M} h) h d\xi\le 0. $
giving symmetry and nonnegativity of $\partial \cQ|_{\underline M} $ on the space $\bH$ defined by 
the square root Maxwellian-weighted norm
\be\label{mnorm}
\|f\|_{\bH}:=\|f \underline M^{-1/2}\|_{L^2(\R^3)}.
\ee
%NOTES: not entirely trivial computation, need that log of Maxwellian is collision invariant... and Q(M,M)=0...
%{Spectral gap} [Carleman] \cite{G} $\Rightarrow$ $E\leq -\eta <0$.

Making the coordinate change
\be\label{change}
\bu= \langle \xi \rangle^{1/2}f, \quad Q(\bu): =  \langle \xi^{-1/2} \rangle^{-1}\cQ(\langle \xi\rangle^{-1/2}\bu), 
\qquad \langle \xi\rangle:=\sqrt{1+|\xi|^2}, 
\ee
and definining multiplication operators $A^0=\langle \xi\rangle^{-1}$ and $A=\xi_1/\langle \xi\rangle$,
we find that \eqref{Boltz} may be put in form \eqref{Relax}, for $\bu\in \bH$, with $A^0$, $A$ evidently
symmetric and bounded, $A^0>0$, and $Q'(\obu)$ symmetric nonpositive
at any equilibrium $\obu= \langle \xi\rangle^{1/2}\underline M$.
By \cite[Cor. 2.4]{MZ}, $Q$ is bounded as a bilinear map on $\bH$.
Moreover, by \cite[Prop. 3.5]{MZ}, $Q'(\obu)$ is negative definite with respect to $\bH$ on its range, this
last being a straightforward consequence of Carleman's theorem \cite{C} that 
$\partial \cQ|_{\underline M} $ acting on $\bH$ may be decomposed as the sum of a multiplication operator
$\nu(\xi)\sim \langle -\xi\rangle$ and a compact operator $K$, whence $Q'(\obu)$ is the sum of a multiplication
operator $\tilde \nu(\xi)\sim -1$ and the compact operator 
$\tilde K=\langle \xi\rangle^{-1/2} K\langle \xi\rangle^{-1/2}$, Weyl's Theorem
thereby implying existence of a spectral gap.

Collecting information, we find that we have reduced to a system of form \eqref{Relax} satisfying (H1)-(H2),
with $\bV:=\langle \xi\rangle^{1/2}(\Range R)^\perp$, $R$ as in \eqref{defR}, $\dim \bV^\perp=5$, and
$\obu=\langle \xi\rangle^{1/2}\underline M$ for any Maxwellian $\underline M$. 
Note that $A$ has no kernel on $\bH$, but essential spectra 
$\xi_1/\langle \xi\rangle \to 0$ as $\xi_1\to 0$: an {\it essential singularity}.
A consequence is that {\it small velocities $\xi_1\to 0$ constitute the main difficulties
in our analysis,} large-velocities issues having been subsumed in the reduction \cite{MZ} to form \eqref{Relax}. 

\subsubsection{Hydrodynamic limit}\label{s:hydro}
The formal Chapman-Enskog expansion \eqref{ce2}, or hydrodynamic limit, being independent of coordinate representation, is the
same in our variables $\bu$, $Q$ as in the standard Boltzmann variables $f$, $\cQ$.
As computed, e.g., in \cite{Ce,LiuYu}, this appears in fluid variables \eqref{fvar} as the 
{\it compressible Navier-Stokes} equations with temperature-dependent viscosity and heat conduction:
\be\tag{cNS}\label{cns}
\begin{aligned}
\rho_t + (\rho v_1)_x&=0,\\
	(\rho v_1)_t + (\rho v_1^2 + p)_x&= ((4/3)\mu v_{1,x})_x ,\\
	(\rho _2)_t + (\rho v_1v_2)_x&= (\mu v_{2,x})_x ,\\
	(\rho _3)_t + (\rho v_1v_3)_x&= (\mu v_{3,x})_x ,\\
(\rho E)_t + (\rho v_1 \rho E + v_1 p)_x&= ( \kappa T_x + (4/3)\mu v_1 v_{1,x} )_{x},\\
\end{aligned}
\ee
where $T$ denotes temperature, with monatomic equation of state $p=\Gamma \rho e$, $T=c_v^{-1} e$, with 
\be\label{mukappa}
\Gamma= 2/3,\quad  c_v=3/4, \quad
\mu=\mu(T)= (5/16)\sqrt{T/\pi} ,\quad \kappa=\kappa(T)=(75/16)\sqrt{T/\pi}.
\ee

As computed in, e.g., \cite{Sm}, the hyperbolic (i.e., lefthand side) part of \eqref{cns} has characteristics
\be\label{bchar}
\lambda_1= v_1-c, \quad \lambda_2=\lambda_3=\lambda_4= v_1, \quad \lambda_5= v_1+c, 
\ee
where $c:=\sqrt{\Gamma(1+\Gamma) e}>0$ denotes sound speed, with ``acoustic modes'' $v_1\pm c$ simple 
 and satisfying \eqref{gnl}, and ``entropic/vorticity modes'' $v_1$ multiplicity three and linearly degenerate in
the sense of Lax \cite{La,Sm} (not addressed here; see \cite{PZ2} for discussion of the linearly degenerate case).

\subsection{Macro-micro decomposition}\label{s:macro}
Next, starting with form \eqref{steady}, (H1)-(H2), coordinatize as in Section \ref{s:ce}
$\bu$ as $(u,v)$, $u=P_{\bV^\perp}\bu$, $v=P_{\bV}\bu$, where $P_{\bV^\perp}$ and $P_{\bV}$ are the orthogonal
projections associated with orthogonal decomposition $\bH=\bV^\perp \oplus \bV$, to obtain the block decomposition
	\be\label{mm}
	\bp A_{11}& A_{12}\\A_{21}& A_{22} \ep \bp u\\v\ep'= \bp 0 & 0\\ 0 & E\ep \bp u\\v\ep +\bp 0\\f\ep,
	\ee
into ``macro'' and ``micro'' variables $u$ and $v$ similarly as in \cite{LiuYu,MZ},
with forcing term $f=B(\bu,\bu)$, where $B$ is a bounded bilinear map and $E<0$ is symmetric negative definite on $\bH$.
The following further reduction greatly simplifies computations later on; {hereafter we take $E=-\Id$}.

\begin{obs}\label{idobs}
By the change of variables $v\to (-E)^{1/2}$ combined with left-multiplication of the $v$-equation by $(-E)^{-1}$, 
we may take without loss of generality $E=\Id$.
\end{obs}

\subsection{Reduction to canonical form}\label{s:canred}
Since $A$ and $P_{|\bV^\perp}$ are self-adjoint on $\bH$, 
$A_{11}=P_{\bV^\perp}A_{|\bV^\perp}$ is self-adjoint on $\bV^\perp$, hence
$\bV^\perp=\ker A_{11}\oplus\im A_{11}$. 
Denote by $P_{\ker A_{11}}$ and $P_{\im A_{11}}$ the associated orthogonal projections onto $\ker A_{11}$ and $\im A_{11}$,
and $\tA_{12}:\bV\to\im A_{11}$ and $T_{12}:\bV\to\ker A_{11}$ the operators defined by $\tA_{12}=P_{\im A_{11}}A_{12}$ and $T_{12}=P_{\ker A_{11}}A_{12}$. 
From the assumption that $A$ is one-to-one, we readily obtain the following; see \cite[Lemma 2.1]{PZ2} for details.

\begin{lemma}\label{r2.1} Assuming (H1)-(H2),
(i) $\ker T_{12}^*=\{0\}$, $\im T_{12}=\ker A_{11}$, $\ker T_{12}\ne\{0\}$, and
(ii) The linear operator $\tA_{11}=(A_{11})_{|\im A_{11}}$ is self-adjoint and invertible on $\im A_{11}$.
\end{lemma}

Introduce now orthogonal subspaces $\bV_1=\im T_{12}^*$ and $\tbV=\ker T_{12}$ decomposing $\bV$,
with associated projectors $P_{\bV_1}$ and $P_{\tbV}$.
Denoting 
\be\label{coordn}
u_1=P_{\ker A_{11}}u, \quad  \tu=P_{\im A_{11}}u, \quad v_1=P_{\bV_1}v, \quad \hbox{\rm and $\tv=P_{\tbV}v$,}
\ee 
and applying $P_{\ker A_{11}}$ and $P_{\im A_{11}}$ to the first equation of \eqref{mm} we obtain 
\begin{equation}\label{u1-tildeu}
T_{12}v'= 0,\quad\tA_{11}\tu'+\tA_{12}\tv'= 0.
\end{equation}
Moreover, by $(A_{21})_{|\ker A_{11}}=T_{12}^*$, $(A_{21})_{|\im A_{11}}=\tA_{12}^*$ the second equation of \eqref{mm} is equivalent to
\begin{equation}\label{inter-eq2}
T_{12}^*u_1'+\tA_{12}^*\tu'+A_{22}v'=Ev+f.
\end{equation}
Since $v_1\in\bV_1=\im T_{12}^*$, from \eqref{u1-tildeu} we conclude $v_1'=0$. In addition, since $\tA_{11}$ is invertible on $\im A_{11}$ by Lemma~\ref{r2.1}(ii), we have $\tu'=-\tA_{11}^{-1}\tA_{12}\tv'$. Summarizing, \eqref{u1-tildeu} is equivalent to
\begin{equation}\label{u1-tildeu2}
v_1'=0,\quad  (\tu +\tA_{11}^{-1}\tA_{12}\tv)'=0.
\end{equation}

Next, taking without loss of generality $E=\Id$, we obtain from \eqref{inter-eq2} evidently 
\begin{equation}\label{u1-1}
T_{12}^*u_1'+P_{\bV_1}(A_{22}-\tA_{12}^*\tA_{11}^{-1}\tA_{12})\tv'=-v_1 +P_{\bV_1}f
\end{equation}
and
\begin{equation}\label{prelim-tilde-v}
P_{\tbV}(A_{22}-\tA_{12}^*\tA_{11}^{-1}\tA_{12})\tv'=-\tv+P_{\tbV}Ev_1+P_{\tbV_1}f.
\end{equation}
From Lemma~\ref{r2.1}(i), $(T_{12}^*)^{-1}$ is well-defined and bounded, hence we obtain from \eqref{u1-1}
\begin{equation}\label{u1-2}
(u_1- \Gamma_1\tv)'= -(T_{12}^*)^{-1}v_1 + (T_{12}^*)^{-1}P_{\bV_1}f,
\quad
\Gamma_0\tv'=\tv+P_{\tbV}f,
\end{equation}
where $\Gamma_1=(T_{12}^*)^{-1}(\tA_{12}^*\tA_{11}^{-1}\tA_{12}-A_{22})\in\mathcal{B}(\bV,\ker A_{11})$ and
\be\label{Gamma0}
\Gamma_0=P_{\tbV}(A_{22}-\tA_{12}^*\tA_{11}^{-1}\tA_{12})_{|\tbV}\in\mathcal{B}(\tbV) 
\; \hbox{\rm is symmetric.}
\ee

Summarizing, we have that \eqref{mm} is equivalent to the system
\be\label{sys1}
(u_1- \Gamma_1\tv)'= (T_{12}^*)^{-1}v_1 + (T_{12}^*)^{-1}P_{\bV_1}f, \quad
(\tu + \tA_{11}^{-1}\tA_{12}\tv)'=0, \quad
v_1'=0, \quad
\Gamma_0\tv'=\tv+P_{\tbV}f.
\ee
By the invertible change of coordinates
\ba\label{coord}
w_c=\Big(  ( u_1 - \Gamma_1\tilde v)^T,
(-(T_{12}^*)^{-1}v_1)^T,
(\tilde u + \tilde A_{11}^{-1}\tilde A_{12}\tilde v)^T\Big)^T ,
\quad
w_h=\tilde v,
\ea
we reduce \eqref{mm} finally to the canonical form of Lemma \ref{l:canon}:
\ba\label{canon2}
w_c'=Jw_c + g_c, \quad
\Gamma_0 w_h' = w_h + g_h,
\ea
where $J=\bp 0 & I_m & 0 \\ 0&0 & 0\\ 0&0&0 \ep$ and
$g_c=\tilde Q_c(w,w)$ and $g_h=\tilde Q_h(w,w)$ are bounded bilinear maps.

\begin{obs}\label{fibers}
We record for later that the tangent subspace $(u,v)=(\zeta, 0)$ to equilibrium manifold $\cE=\{(u,v)\in\bV^\perp\oplus\bV: Q(u,v)=0\}$
is given in coordinates \eqref{coord} by
$w_c=(\zeta_1, 0, \tilde \zeta)$, $w_h=0$,
as can also be seen by computing the subspace of equilibria of \eqref{canon} with $g=(g_c,g_h)=0$.
\end{obs}

\section{Linear resolvent estimates}\label{s:linear}
The starting point for construction of invariant manifolds is the study of the solution operator for the decoupled
linear inhomogeneous equations \eqref{canon2} with arbitrary forcing terms $g_c$, $g_h$.
The ``center,'' $w_c$ equation is of standard finite-dimensional type, so may be treated by usual methods.
Evidently, then, the key issue is treatment of the degenerate ``hyperbolic,'' $w_h$ equation.

\subsection{Symmetric degenerate evolution equations} Consider a degenerate inhomogeneous evolution equation
$(\Gamma_0\partial_x + \Id) w_c=g,$
with $\Gamma_0$ (recalling \eqref{Gamma0} and (H1)-(H2)) symmetric and one-to-one but not boundedly invertible,
with the goal to obtain bounds on the resolvent operator 
\be\label{res}
\mathcal{R}:=(\Gamma_0 \partial_x -\Id)^{-1}.  
\ee
As discussed in the introduction, the inhomogeneous flow $u'+\Gamma_0^{-1}u=0$ possesses generalized exponential dichotomies,
but the resulting bounds on $(\partial_x-\Gamma_0^{-1})^{-1}$
are insufficient to bound the inhomogeneous solution operator
$(\Gamma_0 \partial_x-\Id)^{-1}= (\partial_x -\Gamma_0^{-1})^{-1}\Gamma_0^{-1}$.

That is, \eqref{inhom} represents an interesting new class of symmetric degenerate evolution equations
{for which construction of dichotomies is inadequate to bound the resolvent \eqref{res}.}

{A key observation} of \cite{PZ1} is that $L^2$ bounds may be obtained {\it directly}, using symmetry.
In \cite{PZ1}, we use for technical reasons a frequency domain/Fourier transform formulation 
following \cite{LP2,LP3}; however, this can be seen at formal level through
an a priori energy estimate
\be\label{key}
|\langle u, g\rangle|= |\langle u, \Gamma_0 u'\rangle -\langle u, u\rangle |=
|\langle u, u\rangle| = \|u\|^2 \Rightarrow \|u\|\leq C\|g\|,
\ee
reminiscent of Friedrichs estimates for symmetric hyperbolic PDE,
where $\|\cdot\|$ and $\langle \cdot, \cdot \rangle$ denote $L^2$ norm and inner product;
indeed, one could view \eqref{inhom} as a ``symmetric hyperbolic'' analog for ODE.
As in the PDE setting, the crucial property of symmetry of $\Gamma_0$
is guaranteed by existence of a convex entropy for \eqref{Relax} \cite{CLL}, e.g., the Boltzmann H-Theorem 
as discussed in Section \ref{s:sym}.

\subsection{Details/counterexamples} 
Viewing the constant-coefficient operator $\mathcal{R}$,
\eqref{res}, as a Fourier multiplier with symbol $\sR(\omega)=(i\omega \Gamma_0 -\Id)^{-1}$, 
and computing the uniform estimates
\be\label{Rbds}
|\sR(\omega)|\leq C,\quad 
|\sR'(\omega)|=|-\sR i\Gamma_0 \sR|
\leq C_2(1+|\omega|)^{-1},
\ee
we find by the Mikhlin-Hormander multiplier theorem that $\mathcal{R}$ is bounded on $L^p$, $1<p<\infty$.

Further detail may be obtained by spectral decomposition of $\Gamma_0$, converting 
$(\Gamma_0\partial_x + \Id) w_c=g$ into a family of scalar equations
	$(\alpha_\lambda \partial_x-1)u_\lambda=g_\lambda$, with
$u_\lambda$ the coordinate associated with spectrum $ \alpha_\lambda$ and
$\|u\|_{\bH}^2=\int |u_\lambda|^2d\mu_\lambda$.
The associated (scalar) resolvent operators $\mathcal{R}_\lambda=(\alpha_\lambda \partial_x+1)^{-1}$
have explicit kernels 
\be\label{sR}
R_\lambda (\theta)= \alpha_\lambda^{-1}e^{(\theta)/\alpha_\lambda^{-1}},
\, \theta \alpha_\lambda<0;
\qquad
(\mathcal{R}_\lambda h) (x)= \int_{\R} R_\lambda (x-y)h(y)dy,
\ee
that are evidently integrable with respect to $x$, so bounded coordinate-wise on
any $L^p(\R_+)$.
However, explicit example \cite[Eg. 4.7, p. 23]{PZ1} shows that the full operator $\mathcal{R}$ is not
bounded on $L^\infty(\R,\bH)$ (resp. $L^1(\R,\bH)$); that is, {\it it is not an $L^\infty$ (resp. $L^1$)
multiplier.}
This has the important consequence that our dynamical theory must be carried out in $H^1$ (bounding
$L^\infty$) rather than the usual $C^0(\R)$
setting costing a surprising amount of technical difficulty.

The above shows also that the full resolvent kernel $R(\theta)$ determined by \eqref{sR}, considered
as an operator-valued function from $\bH\to \bH$, {\it is not integrable}, since otherwise $\mathcal{R}$ by standard
convolution bounds would be a bounded multiplier on all $L^p$.
Likewise, the computation
\be\label{supcalc}
|R(\theta)|_{\bH}= \sup_{\alpha_\lambda(\theta<0}
|\alpha_\lambda^{-1}e^{-\theta/\alpha_\lambda^{-1}}| \sim  C/|\theta|
\, \hbox{\rm as $\theta\to 0$}
\ee
shows that $|R(\theta)|_{\bH}$ {\it is not bounded}.
This indicates the delicacy of, and cancellation involved in, the bounds on $\mathcal{R}$
obtained above through energy estimate \eqref{key}/resolvent bounds \eqref{Rbds}.

\br\label{emprmk}
We emphasize that $L^p$ multiplier theory/spectral decomposition is used here only to construct counterexamples, 
our construction of invariant manifolds relying on Parseval's identity.
\er

\subsection{The Banach space setting}
(Following \cite{Z6})
Weighted $L^\infty$ spaces $L^\infty_{r,\xi}$ in velocity 
$\xi$, defined by norms $\|f\|_{L^\infty_{r,\xi}}:=\sup_{\xi\in \R^3}(1+|\xi|)^r M(\xi)^{-1/2}|f(\xi)|$, $r\geq 0$,
where $M(\xi)=e^{-c_0|\xi-v|^2}$ is the Maxwellian corresponding to equilibrium $\obu$,
have been used in the study of Boltzmann's equation in, e.g., \cite{LiuYu,LY3}.
Though resolvent bounds appear more difficult to obtain in this context, we can establish 
that {\it $|R(\theta)|_{L^\infty_{r,\xi}}$ is not bounded,} similarly as in the Hilbert case.

Recall \cite{G} that the linearized collision operator $L$ appearing in the
linearized inhomogeneous steady Boltzmann equation $\xi_1 f'- L f= \tilde g$
may be decomposed as $ \tilde L=\tilde{\nu}(\xi) + \tilde {K}$,
where $\tilde \nu(\xi)$ is a multiplication operator with $\tilde \nu(\xi)\sim \langle\xi\rangle$ 
and $\tilde{K}$ has kernel 
$|\tilde k(\xi,\xi_*)|\leq C|\xi-\xi_*|^{-1}e^{-c|\xi-\xi_*|^2}$,
$(\tilde {K}h)(\xi)=\int_{\R^3} \tilde k(\xi,\xi_*)h(\xi_*)d\xi_*.$
By $\||\xi|^{-1}e^{-c|\xi|^2}\|_{L^1}<\infty$ and standard convolution bounds, 
$\tilde{K}$ is bounded on $L^\infty(\xi)$, hence, by 
$\langle \xi\rangle/\langle \xi-\xi_* \rangle\leq C\langle \xi_*\rangle$, on $L^\infty_{r,\xi}$.
In our coordinates \eqref{change}, $A f'- Q'f=g$,
$A= \frac{\xi_1}{\langle\xi\rangle}$, $Q'=-\nu(\xi)+{K}$, where $\nu(\xi)\sim 1$ and 
${ K}=
\langle\xi\rangle^{-1/2}\hat { K}\langle \xi\rangle^{-1/2}$
is bounded from 
$L^\infty_{r,\xi}\to L^\infty_{r,\xi}$.
The reduced equation $\Gamma_0 u'-Eu=g$
of \eqref{inhom} corresponds to the restriction of $g$ to a finite-codimension subspace $\Sigma$ of ``hyperbolic modes,'' 
where $E:=Q'|_\Sigma<0$ \cite{PZ1}.
That 
\be\label{Rinf}
|R(\cdot)|_{L^\infty_{r,\xi}}=\infty
\ee
thus follows (by contradiction, using
standard convolution bounds) from the following slightly stronger statement.

\begin{lemma}[adapted from \cite{Z6}]\label{claim1}
{The solution of $Au'-Q'(\obu)u=g$, 
with data $g$ valued in a finite-codimension subspace $\Sigma$ of $L^\infty_{r,\xi}$
does not satisfy a uniform bound} $|u|_{L^\infty(x,L^\infty_{r,\xi})}\leq C |g|_{L^1(x,L^\infty_{r,\xi})}.$
\end{lemma}

\begin{proof}
Defining $\mathcal{S}=\big( (\xi_1/\langle \xi\rangle )\partial_x - \nu(\xi) \big)^{-1}$, we have
the explicit solution formula
\be\label{Sform}
(\mathcal{S}g)(x,\xi)= \int_{\R}S_\xi(x-y)g(y,\xi) dy;\qquad
S_\xi(\theta)=(\xi_1/\langle \xi\rangle)^{-1}e^{-\nu(\xi)/(\xi_1/\langle \xi\rangle)^{-1}},
\ee
where scalar kernels $S_\xi(\cdot)$ are integrable, hence $\mathcal{S}$ is bounded on
$L^\infty(x, L^\infty_{r,\xi})= L^\infty_r(\xi, (L^\infty(x)$.
Writing $Au'-Q'(\obu)u=g$ as
$\big( (\xi_1/\langle \xi\rangle )\partial_x - \nu(\xi) \big)u= {K}u + g$, applying 
$\mathcal{S}$, and rearranging, we obtain
$ \mathcal{S}g= u- \mathcal{S}{K}u$, hence $| \mathcal{S}g |_{L^\infty(x,L^\infty_{r,\xi})}\leq C| u |_{L^\infty(x,L^\infty_{r,\xi})}$
by boundedness of $|K|_{L^\infty_{r,\xi}}$, $|\mathcal{S}|_{L^\infty(x,L^\infty_{r,\xi})}$.
Thus,  $|u|_{L^\infty(x,L^\infty_{r,\xi})}\leq C |g|_{L^1(x,L^\infty_{r,\xi})}$ would imply 
$|\mathcal{S}g|_{L^\infty(x,L^\infty_{r,\xi})}\leq C |g|_{L^1(x,L^\infty_{r,\xi})}$, 
or, taking $g\to \delta(x) h(\xi)$, 
$|S(\theta)|_{L^\infty_{r,\xi}}\leq C$ for the full kernel $S$ of $\mathcal{S}$.
But, direct calculation as in \eqref{supcalc} shows 
$|S(\theta)|_{L^\infty_{r,\xi}}\sim |\theta|^{-1}$ as $\theta\to 0$, a contradiction.
\end{proof}

\br\label{LYrmk}
In our notation, the bound asserted in \cite{LiuYu} is
$ |R(\theta )|_{L^\infty_{5/2,\xi}}\leq Ce^{-\beta |x|}, $
in contradiction with \eqref{Rinf}.
We conjecture that $|R(\theta)|_{L^\infty_{r,\xi}} \gtrsim |\theta|^{-1}$ as $\theta\to 0$ similarly as for its principal part $ S(\theta)$,
and similarly as in the Hilbert space setting \eqref{supcalc}, so that
$|R(\cdot )|_{L^\infty_{r,\xi}} \not \in L^p(\R)$ for any $1\leq p\leq \infty$.
\er

\section{$H^1$ stable manifold theorem}\label{s:stable}
We now outline the argument for construction of the stable manifold; for details, see \cite{PZ1}.

\begin{proof}[Proof of Theorem \ref{t1.3}]
For clarity, we first treat the {\it noncharacteristic case} $m=0$, $\dim w_c=\dim \cE=r$, for which \eqref{canon} becomes
$ w_c\equiv \const$, $(\Gamma_0 \partial_x + \Id) w_h= \tilde Q_c(w)$,
and the equation for the stable manifold reduces to $w_c\equiv 0$ and 
$(\Gamma_0 \partial_x + \Id) w_h= B(w_h,w_h)$, $B(w_h,w_h):=\tilde Q_c((0,w_h)$ a bounded bilinear map. 
Inverting, we deduce (see \cite{PZ1}) the fixed-point formulation
\be\label{fixedstable}
u(\tau)= T_S(\tau) \Pi_S u_0 + ( \Gamma_0 \partial_x-\Id)^{-1} B(u,u)(\tau),
\ee
where $\Pi_S$, $T_S$ denote projection and semigroup associated with the stable
subspace of homogeneous flow $\Gamma_0 u'=-u$, so that 
$T_S(\tau) \Pi_S u_0$ is a homogeneous solution with data $\Pi_Su_0$ lying in the
stable subspace at $\tau=0$, and $\Pi_Su(0)=\Pi_Su_0$.
This can be recognized as a concise, frequency-domain version of the usual variation
of constants formula for finite-dimensional ODE.

However, significant new difficulties arise from the fact that, due to the properties
of $(\Gamma_0 \partial_x+\Id)^{-1}$ described in Section \ref{s:linear}, we must carry out the analysis in
weighted $H^1$ rather than standard $L^\infty$ spaces.
For example, for the unbounded formal generator $-\Gamma_0^{-1}$, the $H^1$-stable subspace
is strictly contained in the $L^2$-stable one, so that we must seek a graph not over the
entire stable subspace but only the $H^1$ part, conveniently
conveniently characterized as $\dom(\Pi_S (-\Gamma_0)^{-1/2})$.
Moreover, differentiating the equation gives
$u'(\tau)= T_S'(\tau) \Pi_S (u_0 -B(u(0),u(0)))
( \Gamma_0 \partial_x-\Id)^{-1} B(u,u)'(\tau)$
(noting $u'(0)= -\Gamma_0^{-1} \big( u(0)-B(u(0),u(0)) \big)$) by the equation) 
so that the ``homogeneous term'' involving $T_S'$ lies in $L^2$ 
when $v_0:=u_0 -B(u(0),u(0))$, not $u_0$, lies in the $H^1$-stable subspace.

Our solution is to introduce the {\it modified fixed-point equation}
\be\label{modfixedstable}
u(\tau)= T_S(\tau) \Pi_S\big( v_0-B(u(0),u(0)) \big) + ( \Gamma_0 \partial_x-\Id)^{-1} B(u,u)(\tau)
\ee
parametrized by elements $v_0$ in the $H^1$-stable
subspace, for which the derivative equation is the harmless
$u'(\tau)= T_S'(\tau) \Pi_S v_0+ ( \Gamma_0 \partial_x-\Id)^{-1} B(u,u)'(\tau)$. Observing that
the trace $u\to u(0)$ is bounded on $H^1$ by 1D Sobolev embedding, as is $(\Gamma_0\partial_x + \Id)^{-1}$
by $L^2$-boundedness plus commutation of constant coefficient operators with derivatives, we find that 
\eqref{modfixedstable} is contractive, yielding existence/uniqueness in $H^1$ (and
exponentially weighted $H^1$) norm, and thereby existence of an (exponentially decaying)
stable manifold expressed as a graph over the $H^1$ stable subspace, Fr\'echet-differentiable
from $\dom(-(\Pi_S \Gamma_0)^{-1/2})$ with norm induced by
$(-\Pi_S \Gamma_0)^{-1/2}$ to the full space $\bH$ with its original norm.
A novel aspect is that {the graph lies above the $H^1$-stable subspace not only in unstable
directions, but also in stable directions lying in the stable but not $H^1$-stable subspace.}

In the noncharacteristic case, there is a nontrivial center equation $w_c'=Jw_c + B_c(w,w)$, coupled
to the hyperbolic equation $\Gamma_0 w_h'=-w_h + B_h(w,w)$. 
This may be treated, setting $w=(z,u)$, by the larger fixed-point equation
appending to \eqref{modfixedstable} a standard finite-dimensional $z$ equation:
\ba\label{fixsys}
z(\tau)&= -\int_\tau^{+\infty} e^{J(\tau-\theta)}B_c(w,w)(\theta)d\theta,\\
u(\tau)&= T_S(\tau) \Pi_S\big( v_0-B_h(u(0),u(0)) \big) + ( \Gamma_0 \partial_x-\Id)^{-1} B_h(w,w)(\tau).
\ea
\end{proof}

\begin{proof}[Proof of Corollary \ref{c5.9}]
Because the stable manifold contains the forward orbits of all solutions with $H^1(\R_+,\bH))$ norm sufficiently small,
it contains the orbit on $\R_+$ of $\check{\bu}_M:=\check{\bu}(\cdot + M)$ for $M$ sufficiently large,
whence $\check{\bu}_M\in H^1_{\tilde \nu}(\R_+,\bH)$ by Theorem \ref{t1.3}.
It follows that $e^{\tilde \nu |\cdot|}\bu_M \in H^1(\R_+,\bH)$, hence, by Sobolev embedding, 
$|e^{\tilde \nu |x|}\bu(x)|\leq C$, or
$|\bu(x)|\leq C|e^{\tilde \nu |x|}$, for $x\geq M$.
\end{proof}

\br\label{stablermk}
The key technical points in the above construction are the use of $H^1$ rather than sup norms 
to bound the resolvent, and the ``integration by parts'' parametrization by $v_0$ in \eqref{modfixedstable}.
\er

\section{Existence of a center manifold}\label{s:cm}
Next, we outline the argument for existence of a an $H^1$ center manifold; for details, see \cite{PZ2}.
The translation from standard $C^0$ to $H^1$ framework again introduces interesting new difficulties:
surprisingly, different from those encountered in the stable manifold case.

\begin{proof}[Proof of Theorem \ref{t1.1}]
Following the standard approach to construction of center manifolds \cite{VI,B,Z1}, 
we first replace $\tilde Q$ by a truncated nonlinearity
$N_\eps(w):=\rho(w/\eps)\tilde Q(w)$,
where $\rho$ is a smooth cutoff function equal to $1$ for $|w|\leq 1$ and $0$ for $|w|\geq 2$.
The truncated nonlinearity satisfies bounds
\be\label{Nbds}
|N_\eps| \leq \rmc\eps^2, \quad
|N_\eps'|\leq \rmc\eps, \quad
|N_\eps''| \leq \rmc \qquad
(|\cdot|=|\cdot|_\bH),
\ee
and agrees with the original one locally to $\obu$.

Translating the usual sup-norm approach to the $H^1$ setting, we 
seek solutions to the modified (truncated) equation in a negatively weighted space
$H^1_{-\alpha}$, for $\alpha>0$ sufficiently small.
Similarly as in \eqref{fixedstable}, this yields the fixed point formulation
\be\label{cfixed}
w(\tau)= T_\rmc(\tau) \Pi_\rmc w_0  +\int_0^\tau T_\rmc(\tau -\theta) \Pi_\rmc N_\eps(w(\theta))d\theta
+ ( \Gamma_0 \partial_x+ \Id)^{-1} \Pi_\rmh N_\eps(w)(\tau)
\ee
for solution $w=(w_\rmc^T, w_\rmh^T)^T$, where $\Pi_\rmc$ denotes projection onto the $w_c$ component,
$T_\rmc(\cdot)=e^{J (\cdot)}$ the associated (nondegenerate) flow,
and $\Pi_\rmh$ denotes projection onto the $w_\rmh$ component.

The difficulty in this case is not with the ``homogeneous'' term $T_\rmc(\tau) \Pi_\rmc w_0$ as in the stable
manifold case (since derivatives on $\Sigma_\rmc$ are bounded) nor 
$\int_0^\tau T_\rmc(\tau -\theta) \Pi_\rmc N_\eps(w(y))dy$,
but the formerly harmless $ ( A \partial_x-\Id)^{-1} \Pi_H N_\eps(w,w)(\tau)$, specifically, 
the ``substitution operator'' $\mathcal{N}_\eps: w\to N_\eps(w)$. Bounds \eqref{Nbds}
yield readily that \eqref{cfixed} is contractive in $L^2_{-\alpha}$ and bounded in $H^1_{-\alpha}$, 
$\|f\|_{H^s_{-\alpha}}:=\|e^{-\alpha \langle \cdot \rangle} f(\cdot)\|_{H^s}$,
giving existence and uniqueness of a $C^{0+1/2}$ center manifold $\Pi_\rmc w_0\to w(0)$
via the trace map $w\to w(0)$ and the 1-d Sobolev estimate 
$|f(0)|\leq \|f\|_{L^2_{-\alpha}}^{1/2}\|\partial_x f\|_{L^2_{-\alpha}}$.

However, higher (even Lipshitz) regularity seems to require contraction in $\|\cdot\|_{H^1_{-\alpha}}$,
the difficulty lying in term
$$
 \|\partial_x(N_\eps(v_1)-N_\eps(v_2))\|_{L^2_{-\alpha}}
\sim \| \max_j (|N_\eps''(v_j)||\partial_x v_j|)|v_2-v_1| \|_{L^2_{-\alpha}} 
\sim  \sum_j \| |\partial_x v_j||v_2-v_1| \|_{L^2_{-\alpha}} , 
$$
for which the obvious Sobolev embedding estimate gives 
$  \|v_1-v_2\|_{H^1_{-\alpha}} \sum_j(\int_\R |\partial_x v_j|^2)^{1/2}=+\infty$.

{A key observation} is that, for $0<\alpha_1\ll \alpha \ll \alpha_2\ll 1$, 
\eqref{cfixed} is contractive in the {\it mixed norm}
\be\label{mixed}
\|f\|:=\|f\|_{L^2_{-\alpha}}+ \|\partial_x f\|_{L^2_{-\alpha_2}}
\ee
and bounded in $H^1_{-\alpha_1}$ for $\|w\|_{H^1_{-\alpha_1}}\ll 1$.  For, the Sobolev bound
$$
e^{-2\alpha_2 \langle x \rangle}|f(x)|^2 \leq \|f\|_{L^2_{-\alpha_2}(x,\infty)} \|\partial_x f\|_{L^2_{-\alpha_2}(x,\infty)}
\leq e^{-(\alpha_2-\alpha)\langle x \rangle} \|f\|_{L^2_{-\alpha}(x,\infty)} \|\partial_x f\|_{L^2_{-\alpha_2}(x,\infty)}
	$$
gives	
\ba\label{obs}
		&\|\partial_x(N_\eps(v_1)-N_\eps(v_2))\|_{L^2_{-\alpha_2}}^2 \lesssim
\int_\RR e^{-2\alpha_2\langle x\rangle}|v_1(x)-v_2(x)|^2|\pa_xv_1(x)|^2 dx\\
&\qquad \lesssim \Big(\int_{\RR} e^{-(\alpha_2-\alpha)\langle x\rangle} |\pa_xv_1(x)|^2 dx\Big)\|v_1-v_2\|^2
\lesssim \|\pa_xv_1(x)\|^2_{H^1_{-\alpha_1}} \|v_1-v_2\|^2.
\ea

With this observation, working in norm $\|\cdot\|$, 
we obtain essentially immediately
existence and uniqueness of a global center 
manifold for the truncated equation/ local center manifold 
for the exact equation that is {\it Lipschitz continuous,}
as a graph over the center subspace $\Sigma_\rmc$.
$C^r$ (Fr\'echet) regularity, $r\geq 1$ may then be obtained similarly
as in the finite-dimensional case \cite{VI,B,Z1,HI}, 
by a bootstrap argument,
using a nested sequence of  mixed-weight norms together with
a general result on smooth dependence with 
respect to parameters of a fixed point mapping $y=T(x,y)$ that is 
Fr\'echet differentiable in $y$ from a stronger to a weaker Banach space,
with differential $T_y$ extending to a bounded, contractive map on the 
weaker space \cite[Lemma 2.5, p. 53]{Z1} (\cite[Lemma 3,p. 132]{Z2}).
See \cite[Appendix A]{PZ2}, for further details.
The $H^1$ exponential approximation property (not discussed in \cite{PZ2}) follows by
transcription to the $H^1$ setting of the finite-dimensional argument given in \cite[Step 7, p. 9]{B}.
\end{proof}

\br\label{obsrmk}
The estimate \eqref{obs}, and introduction of norm \eqref{mixed}, we view as the crucial technical points in our construction of center manifolds, and the 
main novelty in this part of the analysis.
\er

\section{Structure of small-amplitude kinetic shocks}\label{s:prof}
Given existence of a center manifold, on may in principle obtain 
an arbitrarily accurate description of near-equilibrium dynamics via formal Taylor expansion/reduction to normal form.
We give here a particularly simple normal form argument 
describing bifurcation of stationary shock profiles from a simple genuinely nonlinear characteristic equilibrium,
adapting more general center manifold arguments of \cite{MaP,MasZ2} in the finite-dimensional case.
Similarly as in \cite{MaP,MasZ2}, the main idea is to use the fact that equilibria are predicted by the Rankine-Hugoniot shock
conditions \eqref{rh} to deduce normal form information from the structure of the Chapman-Enskog approximation \eqref{ce2}.

\begin{lemma}\label{t1.2}
	Let $\obu\in\ker Q$ be an equilibrium satisfying (H1)-(H2).
	In the simple genuinely nonlinear characteristic case \eqref{gnl}, $m=1$, the center manifolds of \eqref{steady} and
	\eqref{ce2} both consist of the union of one-dimensional fibers parametrized by $q\in \R^r$ as in \eqref{rh} and
coordinatized by $u_1$ as in \eqref{coordn}, satisfying an approximate Burgers flow:
without loss of generality 
	\be\label{burgers2}
\tilde q=0,\qquad
	u_1'=\delta^{-1}\big(-q_1 + \Lambda u_1^2/2\big) + O(|u_1|^3 + |q_1||u_1|+ |q_1|^2),
	\ee
	where $\delta:=\obr^T D_* \obr>0$ with $\obr$, $D_*$ as in \eqref{gnl}, \eqref{ce2}.
In particular, under the normalization $\tilde q=0$, there exist local heteroclinic (Lax shock) connections 
for $q_1 \Lambda<0$ between endstates $u_1^\pm \approx \sqrt{-2 q_1/\Lambda}$.
\end{lemma}

\begin{proof}
	First, note that $T_{12}v_1$ in the original coordinates of \eqref{sys1} 
is exactly the first component $q_1$ of $q$ in \eqref{rh}, or $v_1=T_{12}^{-1} q_1$.
By Observation \ref{fibers} and the Implicit Function Theorem, we may take without loss of generality 
$\tilde q=0$ by a shift along equilibrium manifold $\cE$ of the background equilibrium $\obu$.
By \eqref{canon}, therefore, the flow on the $(r+1)$-dimensional center manifold has an $r$-dimensional constant of motion 
\be\label{zetaeq}
\big(w_{c,2},w_{c,3}\big)\equiv (\zeta, \gamma)=\big(-(T_{12}^*)^{-1}v_1, \tilde q\big)= \big(( -T_{12}^*)^{-1}T_{12}^{-1}q_1,0 \big),
\ee
$w$ as in \eqref{coord}, with flow along one-dimensional fibers coordinatized by 
$ w_{c,1}= u_1 - \Gamma_1\tilde v= u_1-\Gamma_1 w_\rmh $
given by the $w_{\rmc,1}$ equation of \eqref{canon}:
\be\label{fibeq}
w_{c,1}'= \zeta+ \phi(w_{c,1},\zeta),
\quad
\phi(w_{c,1},\zeta):=g_{c,1}\big((w_{c,1},\zeta, 0), \Xi(w_{c,1},\zeta, 0)\big)=\mathcal{O}(\|w_{c,1}\|^2,\|\zeta\|^2).
\ee

The factor $ (T_{12}^*)^{-1}T_{12}^{-1}>0$ in term $\zeta= -(T_{12}^*)^{-1}T_{12}^{-1}q_1$ 
is easily recognized as $\delta^{-1}$, where 
$\delta:= T_{12}T_{12}^{*}>0$, or, using $\obr=e_1$,
$\delta= \obr \cdot D_* \obr$ with $D_*$ as in \eqref{ce2}.
Using the fact that  $w_\rmh=\cJ(w_\rmc)=O(|w_\rmc|^2)$ along the center manifold to trade $w_{c,1}$ for
$u_1$ by an invertible coordinate change preserving the order of error terms, we may thus rewrite \eqref{fibeq} as
\be\label{tempfibeqT}
u_1'= \delta^{-1} (-q_1 +  \delta \chi u_{1}^2) +O(|u_{1}|^3 +|u_{1}||q_1|+ |q_1|^2),
\ee
where $\chi$, hence the product $\delta \chi$, is yet to be determined.
On the other hand, performing Lyapunov-Schmidt reduction for the equilibrium problem \eqref{rh}, we obtain the normal form
$$
 0= (-q_1 +  \frac12 \Lambda u_{1}^2) +O(|u_{1}|^3 +|u_{1}||q_1|+ |q_1|^2), 
$$
where $\Lambda$ is as in \eqref{gnl}.
Using the fact that equilibria for \eqref{steady} and \eqref{rh} agree, we find that $\delta \chi$ must be
equal to $\frac12 \Lambda$, yielding a final normal form consisting of the approximate Burgers flow \eqref{burgers}.
A similar computation yields the same normal form for fibers of the center manifold of the formal viscous problem \eqref{ce2};
see also the more detailed computations of \cite{MaP} yielding the same result.

For $q_1 \Lambda>0$, the scalar equation \eqref{burgers} evidently possesses equilibria $\sim \mp\sqrt{2 q_1/\Lambda}$,
connected (since the equation is scalar) by a heteroclinic profile.
Since $\sgn u_1'=-\sgn \Lambda$ for $u_1$ between the equilibria, so that
$\big(\lambda(u)\big)'\sim \Lambda u_1'$ has sign of $-\Lambda^2<0$,
the connection is in the direction of decreasing characteristic $\lambda(u)$, corresponding to a Lax-type solution
of \eqref{rh} (cf. \cite{MaP,MasZ2}).
\end{proof}

\br\label{lamrmk}
Using $\lambda(u_1)\sim \Lambda u_1$, we may rewrite \eqref{burgers2} as \eqref{burgers} as in the introdiuction,
eliminating the $\tilde q$-dependent term $\Lambda$.
However, the ``effective viscosity'' $\delta$ remains dependent on $\tilde q$.
\er

Having determined the normal form \eqref{burgers}, we establish closeness
of profiles of \eqref{steady} and \eqref{ce2} by comparing 
their $u_1$ coordinates, separately, to an exact Burgers shock, then showing that differences in remaining, slaved,
coordinates, since vanishing at both endstates, are negligibly small.

\begin{lemma}[\cite{Li,PlZ}]\label{lburgers}
	Let $\eta\in \R^1$ be a heteroclinic connection of an approximate Burgers equation
	\be\label{approx}
	\delta \eta'= \frac12 \Lambda(-\eps^2 +  \eta^2) +  S(\eps,\eta),
	\quad
	S=O(|\eta|^3+|\eps|^3)\in C^{k+1}(\R^2), \quad k\geq 0,
	\ee
	and $\bar \eta:= -\eps\tanh(\Lambda \eps x/2\delta)$ a connection of the exact Burgers equation
	$\delta\bar \eta'= \frac12 \Lambda(-\eps^2 +  \bar \eta^2)$.
	Then,
	\ba\label{etacomp}
 |\eta_{\pm}-\bar{\eta}_{\pm}| &\leq C\epsilon^2,\\
	|\partial_x^k \big(\bar \eta - \bar \eta_\pm)(x)| &\sim  \eps^{k+1}e^{-\delta  \eps|x|},
	\quad x\gtrless 0, \quad \delta>0,\\
	\big|\partial_x^k \big((\eta- \eta_\pm)- (\bar \eta - \bar \eta_\pm)\big)(x)\big| &\le C \eps^{k+2}e^{-\delta  \eps|x|},
	\quad x\gtrless 0,
	\ea
	uniformly in $\eps>0$, where $\eta_\pm:=\eta(\pm \infty)$, $\bar \eta_\pm:=\bar \eta(\pm \infty)=\mp\eps$ denote
	endstates of the connections.
\end{lemma}
\begin{proof}(From \cite{PZ2}, following \cite{Li})
Rescaling ${\eta} \to {\eta}/\epsilon, {x} \to \Lambda \epsilon \tilde{x}/\beta$,
we obtain the blowup equations
$$
\eta'=\frac{1}{2}(\eta^2-1)+\epsilon \tilde{S}(\eta,\epsilon)
\quad \tilde{S} \in C^{k+1}(\mathbb{R}^2)
$$
and $\bar{\eta}'=\frac{1}{2}(\bar{\eta}^2-1)$, for which estimates \eqref{approx} translate to
	\ba\label{scaleetacomp}
 |\eta_{\pm}-\bar{\eta}_{\pm}| &\leq C\epsilon,\\
	|\partial_x^k (\bar \eta - \bar \eta_\pm)(x)| &\sim  C\eps^{k}e^{-\theta  |x|},
	\quad x\gtrless 0, \quad \theta>0,\\
	|\partial_x^k \big((\eta- \eta_\pm)- (\bar \eta - \bar \eta_\pm)\big)(x)| &\le C \eps^{k+1}e^{-\theta  |x|},
	\quad x\gtrless 0.
	\ea
The estimates \eqref{scaleetacomp} follow readily from the implicit function theorem and stable manifold theorems together with
smooth dependence on parameters of solutions of ODE, giving the result.
\end{proof}

Setting $q_1=\Lambda \eps^2/2$, and either $\eta=u_{REL,1}$ or $\eta=u_{CE,1}$, we obtain approximate Burgers
equation \eqref{approx}, and thereby estimates \eqref{etacomp} relating $\eta=u_{REL,1}$, $u_{CE,1}$ to an exact
Burgers shock $\bar \eta$.

\begin{corollary}[\cite{PZ2}]\label{ctri}
Let $\obu\in\ker Q$ be an equilibrium satisfying (H1)-(H2),
in the noncharacteristic case \eqref{gnl}, and $k$ and integer $\geq 2$.
Then, local to $\obu$ ($\bar u$), each pair of points $u_\pm$ corresponding to a standing Lax-type shock of \eqref{rh}
has a corresponding viscous shock solution $u_{CE}$ of \eqref{ce2} and
relaxation shock solution $\bu_{REL}=(u_{REL},v_{REL})$ of \eqref{steady}, satisfying for all $j\leq k-2$:
	\ba\label{u1comp}
	|\partial_x^j ( u_{REL,1} - u_{REL,1}^\pm )(x)| &\sim C \eps^{j}e^{-\theta  |x|},
	\quad x\gtrless 0, \quad \theta>0,\\
	|\partial_x^j (u_{REL,1}-u_{CE,1})(x)| &\le C \eps^{j+1}e^{-\theta  |x|},
	\quad x\gtrless 0.
	\ea
\end{corollary}

\begin{proof}
	Immediate, by \eqref{scaleetacomp}, Lemma \ref{lburgers} and the triangle inequality, together with the observation that,
	as equilibria of \eqref{ce2} and \eqref{steady}, hence solutions of \eqref{rh},
	endstates $u_{REL,1}^\pm=u_{CE,1}^\pm$ agree.
\end{proof}

\begin{proof} [Proof of Corollary \ref{c1.3}](\cite{PZ2})
	Noting that the $\im A_{11}$ and $\bV$ components of $\bu_{REL}$ are the $C^1$ functions $\Psi(u_{REL,1})$, $\Phi(u_{REL,1})$ of $u_{REL,1}$ along
	the fiber \eqref{burgers}, we obtain \eqref{finalbds}(iii) immediately from \eqref{u1comp}(i).
	Denote by $\Psi_{CE}$ the map describing the dependence of $\im A_{11}$ component of $u_{CE}$ on $u_{CE,1}$
	on the corresponding fiber of \eqref{ce2}.
	Noting that $\Psi-\Psi_{CE}$ and $\Phi- v_*$ both vanish at the endstates $u_{REL,1}^\pm$, we have by smoothness
	of $\Psi$, $\Psi_{CE}$, $\Phi$, $v_*$ that
	$$
	|\Psi-\Psi_{CE}|, \, |\Psi-v_*|=\mathcal{O}(|u_{REL,1}- u_{REL,1}^+|,|u_{REL,1}- u_{REL,1}^-|),
	$$
	giving \eqref{finalbds}(i)-(ii) by \eqref{u1comp}(i)-(ii).
\end{proof}

\br\label{lastrmk}
Applied to Boltzmann's equation, Corollary \ref{c1.3} yields existence/convergence to hydrodynamic shock
profiles in the square root Maxwellian-weighted norm \eqref{mnorm}.
Using a bootstrap argument analogous to that of \cite[Prop. 3.1]{MZ},
one can show \cite[Prop. 1.8]{PZ2} that the center manifold of Theorem \ref{t1.1}
lies in the stronger spaces determined by near-Maxwellian weighted norms
$\|f\|_{\bH^s}:=\|f \underline M^{-s}\|_{L^2(\R^3)}$, $1/2\leq s<1$, yielding further information on
localization of velocity in small-amplitude shock profiles.
This, and the streamlined proof of existence above, are the main novelties in our 
treatment by center manifold techniques
of existence and structure of kinetic shocks.
\er


\begin{thebibliography}{GMWZ7}

\bibitem{AM1} A.\ Abbondandolo, P.\ Majer, \textit{Ordinary differential operators in Hilbert spaces and Fredholm pairs}, Math. Z. \textbf{243}
(2003) 525--562.

\bibitem{AM2} A.\ Abbondandolo, P.\ Majer, \textit{Morse homology on Hilbert spaces}, Comm. Pure Appl. Math. \textbf{54} (2001) 689--760.

%NOTE: existence of finite dim profiles for quasilinear hyp with singularity at a specific x... somewhat std singular ode
%\bibitem{DY} A. Dressler and W.-A. Yong, {\textit Existence of Traveling-Wave Solutions for Hyperbolic Systems of Balance Laws,} Arch. Rational Mech. Anal. \textbf{182} (2006) 49--75.

\bibitem{Ba} B. Barker,
{\it  Numerical proof of stability of roll waves in
the small-amplitude limit for inclined thin film flow,}
   J. Diff. Eq.. 257 no. 8 (2014) 2950--2983.

\bibitem{BZ} B. Barker and K. Zumbrun, 
{\it Numerical proof of stability of viscous shock profiles,} to appear, Math. Models Meth. Appl. Sci.

\bibitem{BGK} H.\ Bart, I.\ Gohberg, M.\ A.\ Kaashoek,\textit{
Wiener-Hopf factorization, inverse Fourier transforms and exponentially dichotomous operators}.
J. Funct. Anal. \textbf{68} (1986), no. 1, 1-42.

\bibitem{BR} G.\ Boillat and T.\  Ruggeri,
	{\it On the shock structure problem for hyperbolic system of balance laws and convex entropy,}
Continuum Mechanics and Thermodynamics 10 (5), 285-292.

\bibitem{B} A. Bressan,
\textit{A Tutorial on the Center Manifold Theorem}, Appendix A, Hyperbolic systems of balance laws, Lecture Notes in
Math., vol. 1911, Springer-Verlag, 2007.

\bibitem{CN} R. Caflisch and B. Nicolaenko,
\textit{ Shock profile solutions of the Boltzmann equation,}
Comm. Math. Phys.  86  (1982), no. 2, 161--194.

\bibitem{C} T. Carleman,
{\it Sur la theorie des equations integrales et ses applications}, 
Verhandl. des Internat. Math. Kong., I, Zurich (1932) 138-151.

\bibitem {Carr} J. Carr,
{\it Applications of centre manifold theory.} Applied Mathematical Sciences, 35. Springer-Verlag, New York-Berlin, 1981. vi+142 pp. ISBN: 0-387-90577-4.

\bibitem{Ce} C. Cercignani,
{\it The Boltzmann equation and its applications,}
Applied Mathematical Sciences, 67. Springer-Verlag, New York, 1988. xii+455 pp. ISBN: 0-387-96637-4.

\bibitem{CLL} Gui Qiang Chen, C. David Levermore, and Tai-Ping Liu,
{\it Hyperbolic conservation laws with stiff relaxation terms and entropy,} Comm. Pure Appl. Math. \textbf{47} (1994) 787--830.

\bibitem{GZ} R.A. Gardner and K. Zumbrun,
{\it The gap lemma and geometric criteria for instability of viscous shock
profiles,} Comm. Pure Appl. Math. \textbf{51} (1998), no. 7, 797--855.

\bibitem{Gl} R. Glassey,
{\it The Cauchy problem in kinetic theory,}
Society for Industrial and Applied Mathematics (SIAM),
Philadelphia, PA, 1996. xii+241 pp. ISBN: 0-89871-367-6.

\bibitem{G} H. Grad,
	{\it Asymptotic theory of the Boltzmann equation. II,}
 1963  Rarefied Gas Dynamics (Proc. 3rd Internat. Sympos., Palais de l'UNESCO, Paris, 1962), Vol. I  pp. 26--59 Academic Press, New York.

\bibitem{HI} M. Haragus and G. Ioos,
{\it Local bifurcations, center manifolds, and normal forms
in infinite-dimensional dynamical systems,} Universitext. Springer-Verlag London, Ltd., London; EDP Sciences, Les Ulis, 2011. xii+329 pp. ISBN: 978-0-85729-111-0; 978-2-7598-0009-4.

\bibitem{HLyZ} J. Humpherys, G. Lyng, and K. Zumbrun,
{\it Multidimensional stability of large-amplitude Navier-Stokes shocks,} preprint; arxiv:1603.03955.

\bibitem{K} S. Kawashima,
{\it Systems of a hyperbolic--parabolic composite type,
 with applications to the equations of magnetohydrodynamics},
thesis, Kyoto University (1983).

\bibitem{LP2}  Y.\ Latushkin, A.\ Pogan, \textit{The dichotomy theorem for evolution bi-families}. J. Diff. Eq. \textbf{245} (2008), no. 8, 2267--2306.

\bibitem{LP3}  Y.\ Latushkin, A.\ Pogan, \textit{The Infinite Dimensional Evans Function}. J. Funct Anal., \textbf{268} (2015), no. 6, 1509--1586.

\bibitem{La} P.D. \ Lax,
\textit{Hyperbolic systems of conservation laws and the mathematical theory of shock waves},
Conference Board of the Mathematical Sciences Regional Conference
Series in Applied Mathematics, No. 11. Society for Industrial and Applied Mathematics, Philadelphia, Pa., 1973. v+48 pp.

\bibitem{Li} Y. Li,
	{\it Scalar Green function bounds for instantaneous shock location and one-dimensional
	stability of viscous shock waves,} to appear, Quart. App. Math. 

\bibitem{L} T.-P. Liu, \textit{Hyperbolic conservation laws with relaxation,}
 Comm. Math. Phys. 108 (1987), no. 1, 153-175.

\bibitem{LiuYu2} T. P.\ Liu, S. H.\ Yu,
{\textit Boltzmann equation: micro-macro decompositions and positivity of shock profiles},
Comm. Math. Phys. 246 (2004), no. 1, 133--179.

\bibitem{LiuYu} T. P.\ Liu, S. H.\ Yu, \textit{Invariant Manifolds for Steady Boltzmann
Flows and Applications}, Arch. Rational Mech. Anal. \textbf{209} (2013) 869--997.

\bibitem{LY3} T.-P. Liu and S.-H. Yu, 
{\it The Green's function and large-time behavior of solutions for the one-dimensional Boltzmann equation,}
Comm. Pure Appl. Math. 57  (2004),  no. 7, 841--876. 

\bibitem{Mallet-Paret} J.\ Mallet-Paret, \textit{The Fredholm alternative for functional-differential equations of mixed type}, J. Dyn. Diff.
Eq. \textbf{11} (1999) 1--47.

\bibitem {MaP} A. Majda and R. Pego,
{\textit Stable viscosity matrices for systems of conservation laws},
J. Diff. Eqs. \textbf{56} (1985) 229--262.

\bibitem{MasZ2} C.\ Mascia, K.\ Zumbrun, \textit{Pointwise Green's function bounds and stability of relaxation shocks}. Indiana Univ. Math. J. \textbf{51} (2002), no. 4, 773--904.

\bibitem{Mi} A. Mielke,
{\it Reduction of quasilinear elliptic equations in cylindrical domains with applications,}
Math. Methods Appl. Sci. 10 (1988) 51--66.

\bibitem{MZ}  G.\ M\'etivier, K.\ Zumbrun, {\it Existence and sharp localization in velocity of small-amplitude Boltzmann shocks}, Kinet. Relat. Models \textbf{2} (2009), no. 4, 667-705.

\bibitem{Pe} R.L. Pego, 
{\it Stable viscosities and shock profiles for systems of 
conservation laws,} Trans. Amer. Math. Soc. 282 (1984) 749--763.


\bibitem{PSS} D.\ Peterhof, B.\ Sandstede, A.\ Scheel, \textit{Exponential dichotomies for solitary-wave solutions of semilinear elliptic equations on infinite cylinders}, J. Diff. Eq. \textbf{140} (1997) 266--308.

\bibitem{RobSal} J.\ Robbin, D.\ Salamon, \textit{The spectral flow and the Maslov index}, Bull. London Math. Soc. \textbf{27} (1995) 1--33.

\bibitem{BjornSand} B.\ Sandstede, \textit{Stability of traveling waves}, in: Handbook of Dynamical Systems, vol. 2, North-Holland, Amsterdam, 2002, pp. 983--1055.

\bibitem{SS1} B.\ Sandstede, A.\ Scheel, \textit{On the structure of spectra of modulated traveling waves}, Math. Nachr. \textbf{232} (2001) 39--93.

\bibitem{SS2} B.\ Sandstede, A.\ Scheel, \textit{Relative Morse indices, Fredholm indices, and group velocities}, Discrete Contin. Dyn.  Syst. A \textbf{20} (2008) 139--158.

\bibitem{PlZ} R.\ Plaza and K.\ Zumbrun,
{\it Evans function approach to spectral stability of small-amplitude shock profiles,}
Discrete Contin. Dyn. Syst. 10 (2004) 885--924.

\bibitem{PZ1} A.\ Pogan and K.\ Zumbrun,
\textit{Stable manifolds for a class of degenerate evolution equations and exponential decay of kinetic shocks,} 
preprint; arxiv: 1607.03028.

\bibitem{PZ2} A.\ Pogan and K.\ Zumbrun,
\textit{Center manifolds of degenerate evolution equations and existence of small-amplitude kinetic shocks,}
preprint; arxiv: 1612.05676.

\bibitem{Sm} J. Smoller,
\textit{Shock waves and reaction--diffusion equations,}
Second edition, Grundlehren der Mathematischen Wissenschaften
[Fundamental Principles of Mathematical Sciences], 258. Springer-Verlag,
New York, 1994. xxiv+632 pp. ISBN: 0-387-94259-9.

\bibitem {VI} A. Vanderbauwhede and G. Iooss,
{\it Center manifold theory in infinite dimensions,}
Dynamics reported: expositions in dynamical systems,  125--163, Dynam. Report.
Expositions Dynam. Systems (N.S.), 1, Springer, Berlin, 1992. 

\bibitem{Y} W.-A. Yong {\it Basic structures of hyperbolic relaxation systems}, Proc. Roy. Soc.
Edinburgh Sect. A 132 (2002), no. 5, 1259--1274.

\bibitem{Z1} K. Zumbrun,
\textit{Conditional stability of unstable viscous shocks},
J. Diff. Eq. \textbf{247}  (2009),  no. 2, 648--671.

\bibitem{Zode} K. Zumbrun,
\textit{Ordinary differential equations},
Lecture notes for graduate ODE, Indiana University (2009).

\bibitem{Z2} K. Zumbrun,
{\it Multidimensional stability of planar viscous shock waves,}
Advances in the theory of shock waves, 307--516,
Progr. Nonlinear Differential Equations Appl., 47, Birkh\"auser Boston,
Boston, MA, 2001.

\bibitem{Z3} K. Zumbrun, {\it Stability of large-amplitude shock
waves of compressible Navier--Stokes equations,}
with an appendix by Helge Kristian Jenssen and Gregory Lyng,
in Handbook of mathematical fluid dynamics. Vol. III,  311--533,
North-Holland, Amsterdam, (2004).

\bibitem{Z4} K. Zumbrun, {\it Planar stability criteria for viscous shock waves of systems with real viscosity,}
Hyperbolic systems of balance laws, 229--326, Lecture Notes in Math., 1911, Springer, Berlin, 2007.

\bibitem{Z5} K. Zumbrun,
{\it Stability and dynamics of viscous shock waves,} Nonlinear conservation laws
and applications, 123--167, IMA Vol. Math. Appl., 153, Springer, New York, 2011.

\bibitem{Z6} K. Zumbrun,
	{\it $L^\infty$ resolvent estimates for steady Boltzmann's equation,} preprint.
	%TODO: arxiv listing...

\bibitem{ZH} K. Zumbrun and P. Howard,
{\it Pointwise semigroup methods and stability of viscous shock waves}.
Indiana Mathematics Journal V47 (1998), 741--871;
Errata, Indiana Univ. Math. J.  \textbf{51}  (2002),  no. 4, 1017--1021.

\bibitem{ZS} K. Zumbrun and D. Serre,
{\it Viscous and inviscid stability of multidimensional
planar shock fronts,} Indiana Univ. Math. J. \textbf{48} (1999) 937--992.

\end{thebibliography}
\end{document}